\documentclass{article}
\usepackage{amsfonts}
\usepackage{amsmath}

\newtheorem{theorem}{Theorem}

\newtheorem{conjecture}[theorem]{Conjecture}
\newtheorem{corollary}[theorem]{Corollary}

\newtheorem{definition}[theorem]{Definition}
\newtheorem{example}[theorem]{Example}

\newtheorem{lemma}[theorem]{Lemma}

\newtheorem{proposition}[theorem]{Proposition}

\newenvironment{proof}[1][Proof]{\textbf{#1.} }{\ \rule{0.5em}{0.5em}}
\input{tcilatex}
\begin{document}

\title{A Graph Theoretic Analysis of Leverage Centrality}
\author{Roger Vargas, Jr.\thanks{%
Mathematics and Statistics, Williams College, Williamstown, MA 01267 USA} \
\ \ \ Abigail Waldron\thanks{%
Mathematics Department, Presbyterian College, Clinton, SC 29325} \ \ Anika
Sharma\thanks{%
Department of Computer Science and Department of Mathematics, University of
Buffalo, Buffalo, NY 14260} \ \ \ \ \ \  \and Rigoberto Fl\'{o}rez\thanks{%
Department of Mathematics and Computer Science, The Citadel, Charleston, SC
29409}\ \ \ \ Darren A. Narayan\thanks{%
School of Mathematical Sciences, Rochester Institute of Technology,
Rochester, NY 14623-5604}}
\maketitle

\begin{abstract}
In 2010, Joyce et. al defined the leverage centrality of vertices in a graph
as a means to analyze functional connections within the human brain. In this
metric a degree of a vertex is compared to the degrees of all it neighbors.
We investigate this property from a mathematical perspective. We first
outline some of the basic properties and then compute leverage centralities
of vertices in different families of graphs. In particular, we show there is
a surprising connection between the number of distinct leverage centralities
in the Cartesian product of paths and the triangle numbers.
\end{abstract}

\section{Introduction}

In a social network people influence each other and those with lots of
friends often have more leverage (or influence) than those with fewer
friends. However the true influence of a person not only depends on the
number of friends that they have, but also on the number of friends that
their friends have. A person that is well connected can pass information to
many friends, but if their friends are also receiving information from
others, their influence on others is lessened. The extreme cases of
influence occurs with a person who has a large number of friends, and for
each of the friends, their only source of information is the original
person. In this situation, the original person has the highest possible
influence and all of the others have the lowest possible influence.

The level of influence can be quantified by a property defined by Joyce et
al. \cite{Joyce2010} known as \textit{leverage centrality}. We recall that
the degree of a vertex $v$ is the number of edges incident to $v$ and is
denoted $\deg (v)$. We next give a formal definition of leverage centrality 
\cite{Joyce2010}.

\begin{definition}
(leverage centrality) \label{def: leverage centrality} Leverage centrality
is a measure of the relationship between the degree of a given node $v$ and
the degree of each of its neighbors $v_{i}$, averaged over all neighbors $%
N_{v}$, and is defined as shown below: 
\begin{equation*}
l(v)=\frac{1}{\deg (v)}\sum_{v_{i}\in N_{v}}\frac{\deg (v)-\deg (v_{i})}{%
\deg (v)+\deg (v_{i})}\text{.}
\end{equation*}
\end{definition}

This property was used by Joyce et al. \cite{Joyce2010} in the analysis of
functional magnetic resonance imaging (fMRI) data \cite{Joyce2010} and has
also been applied to real-world networks including airline connections,
electrical power grids, and coauthorship collaborations \cite{LC2}. However
despite these studies leverage centrality has yet to be explored from a
mathematical standpoint. The formula gives a measure of the relationship
between a vertex and its neighbors. A positive leverage centrality means
that this vertex has influence over its neighbors, where as a negative
leverage centrality indicates that a vertex is being influenced by its
neighbors.

\noindent We begin with an elementary result involving the bounds of
leverage centrality (Li et al. \cite{LC2}).

\begin{lemma}
\label{R1}Let $G$ be a graph with $n$ vertices. For any vertex $v$, $%
\left\vert l(v)\right\vert \leq 1-\frac{2}{n}$. Furthermore, these bounds
are tight in the cases of stars and complete graphs.
\end{lemma}

We note that the bounds are also tight for regular graphs.

There exist graphs $G$ where the leverage centrality of all vertices is
equal and where the leverage centrality of vertices is distinct. It is clear
that if $G$ is a regular graph than $l(v)=0$ for every $v\in G$. We give an
example below of a graph that has distinct leverage centralities.

$%
\begin{array}{cc}
\FRAME{itbpFU}{2.8478in}{2.6178in}{0in}{\Qcb{}}{}{Figure}{\special{language
"Scientific Word";type "GRAPHIC";maintain-aspect-ratio TRUE;display
"USEDEF";valid_file "T";width 2.8478in;height 2.6178in;depth
0in;original-width 2.8055in;original-height 2.5763in;cropleft "0";croptop
"1";cropright "1";cropbottom "0";tempfilename
'OJXR2T0O.wmf';tempfile-properties "XPR";}} & 
\begin{tabular}[b]{|l|l|}
\hline
0 & $\ \frac{1}{5}\left( \frac{5-4}{5+4}\right) =$ $\frac{1}{45}$ \\ \hline
1 & $\frac{1}{5}\left( \frac{5-6}{5+6}\right) =\allowbreak -\frac{1}{55}$ \\ 
\hline
2 & $0$ \\ \hline
3 & $\frac{1}{6}\left( 3\left( \frac{6-5}{6+5}\right) +\left( \frac{6-3}{6+3}%
\right) \right) =\allowbreak \frac{10}{99}$ \\ \hline
4 & $\frac{1}{4}\left( \left( \frac{4-6}{4+6}\right) +2\left( \frac{4-5}{4+5}%
\right) +\left( \frac{4-3}{4+3}\right) \right) =\allowbreak -\frac{22}{315}$
\\ \hline
5 & $\frac{1}{3}2\left( \frac{3-6}{3+6}\right) =\allowbreak -\frac{2}{9}$ \\ 
\hline
6 & $\frac{1}{6}\left( 2\left( \frac{6-5}{6+5}\right) +\left( \frac{6-4}{6+4}%
\right) +\left( \frac{6-3}{6+3}\right) \right) =\allowbreak \frac{59}{495}$
\\ \hline
7 & $\frac{1}{5}\left( 2\left( \frac{5-6}{5+6}\right) +\left( \frac{5-4}{5+4}%
\right) \right) =\allowbreak -\frac{7}{495}$ \\ \hline
8 & $\frac{1-6}{1+6}=\allowbreak -\frac{5}{7}$ \\ \hline
9 & $\frac{1}{6}\left( 3\left( \frac{6-5}{6+5}\right) +\left( \frac{6-1}{6+1}%
\right) \right) =\allowbreak \frac{38}{231}$ \\ \hline
\end{tabular}%
\end{array}%
$

\begin{center}
Figure 1. A graph with distinct leverage centralities\medskip
\end{center}

Intuitively one would think that the sum of the leverage centralities over a
graph would be zero. This is in fact the case when a graph is regular.
However, for non-regular graphs the sum of leverage centralities is
negative. This arises since each edge between two vertices of different
degrees contributes a negative amount to the sum of the leverage
centralities. Let $G$ be the graph $K_{3}$ with a pendant edge (see Figure
2).

\FRAME{dtbpFU}{3.9583in}{2.4699in}{0pt}{\Qcb{Figure 2. Calculating leverage
centrality}}{}{Figure}{\special{language "Scientific Word";type
"GRAPHIC";maintain-aspect-ratio TRUE;display "USEDEF";valid_file "T";width
3.9583in;height 2.4699in;depth 0pt;original-width 6.0485in;original-height
3.7567in;cropleft "0";croptop "1";cropright "1";cropbottom "0";tempfilename
'OJXR2T0P.wmf';tempfile-properties "XPR";}}

Then $l(v_{1})=\frac{1}{2}\left( \frac{2-3}{2+3}\right) +\frac{1}{2}(\frac{%
2-2}{2+1})$; $l(v_{2})=\frac{1}{3}\left( \frac{3-1}{3+1}\right) +\frac{1}{3}(%
\frac{3-2}{3+2})+\frac{1}{3}(\frac{3-2}{3+2})$; $l(v_{3})=\frac{1}{1}\left( 
\frac{1-3}{1+3}\right) $; and $l(v_{4})=\frac{1}{2}\left( \frac{2-3}{2+3}%
\right) +\frac{1}{2}(\frac{2-2}{2+2})$. We can regroup the sum to be $%
\dsum\limits_{v_{i}\in G}l(v_{i})=\frac{1}{2}\left( \frac{2-3}{2+3}\right) +%
\frac{1}{2}(\frac{2-2}{2+1})+\frac{1}{3}\left( \frac{3-1}{3+1}\right) +\frac{%
1}{3}(\frac{3-2}{3+2})+\frac{1}{3}(\frac{3-2}{3+2})+\frac{1}{1}\left( \frac{%
1-3}{1+3}\right) +\frac{1}{2}\left( \frac{2-3}{2+3}\right) +\frac{1}{2}(%
\frac{2-2}{2+2})$

$=\left( \frac{1}{2}\left( \frac{2-3}{2+3}\right) +\frac{1}{3}(\frac{3-2}{3+2%
})\right) +\left( \frac{1}{3}\left( \frac{3-1}{3+1}\right) +\frac{1}{1}%
\left( \frac{1-3}{1+3}\right) \right) +\left( \frac{1}{3}(\frac{3-2}{3+2})+%
\frac{1}{2}\left( \frac{2-3}{2+3}\right) \right) +\left( \frac{1}{2}(\frac{%
2-2}{2+1})+\frac{1}{2}(\frac{2-2}{2+1})\right) $

Since the first three parts are negative and the last part is zero, the sum
must be negative.

\begin{proposition}
\label{negative}For any graph $G$, $\dsum\limits_{v\in G}l(v)\leq 0$.
\end{proposition}

\begin{proof}
If $G$ is a regular graph, then $l(v)=0$ for all $v$, and hence $%
\dsum\limits_{v\in G}l(v)=0$. If $G$ is not regular, there must exist an
edge $e$ with end vertices $u$ and $v$ where $d(u)>d(v)$. We note that the
contribution of each edge $uv$ to the sum of the leverage centralities is $%
\frac{1}{d(v)}\left( \frac{d(u)-d(v)}{d(u)+d(v)}\right) -\frac{1}{d(u)}%
\left( \frac{d(u)-d(v)}{d(u)+d(v)}\right) <0$. Hence for a non-regular
graph, the sum of the leverage centralities is $\dsum\limits_{v\in
G}l(v)=\dsum\limits_{(u,v)\in G}\frac{1}{d(v)}\left( \frac{d(u)-d(v)}{%
d(u)+d(v)}\right) -\frac{1}{d(u)}\left( \frac{d(u)-d(v)}{d(u)+d(v)}\right)
<0 $.
\end{proof}

\section{Vertices with positive / negative leverage centrality}

A vertex of lowest degree cannot have a positive leverage centrality and a
vertex of highest degree cannot have a negative leverage centrality. However
it is possible to have all the vertices in a graph except for one to have
negative leverage centrality, or all but one have positive leverage
centrality. The star graph $K_{1,n-1}$ has $n-1$ vertices with negative
leverage centrality. We show in the next theorem there exist graphs where
there are $n-1$ vertices with positive leverage centrality.

\begin{theorem}
The maximum number of vertices with positive leverage centrality is $n-1$.
\end{theorem}

\begin{proof}
Since the sum of leverage centralities over all vertices in a graph is less
than or equal to zero, it is impossible for a graph to have $n$ vertices
with positive leverage centrality.\ Let $G$ be a graph with vertices $%
v_{1},...,v_{n}$, where $n\geq 11$, and edges $\{v_{i}v_{j}$ $|$ $1\leq
i<j\leq n-4\}\cup \{v_{i}v_{j}$ $|$ $1\leq i\leq n-4$ and $n-3\leq j\leq
n-1\}\cup \{v_{i}v_{n}$ $|$ $n-3\leq i\leq n-1\}$. We note that $\deg
(v_{i})=n-2$ for $1\leq i\leq n-4$, $\deg (v_{i})=n-3$, for $n-3\leq i\leq
n-1$, and $\deg (v_{n})=3$. Then $l\left( v_{i}\right) >0$ for all $1\leq
i\leq n-4$ since these vertices have the largest degree in $G$. Then for $%
n-3\leq i\leq n-1$, $l(v_{i})=\frac{1}{n-3}\left( \left( n-4\right) \frac{%
(n-3)-(n-2)}{(n-3)+(n-2)}+\frac{(n-3)-3}{(n-3)+3}\right) =\allowbreak \frac{1%
}{n\left( 2n-5\right) }\left( n-10\right) $. Here $l(v_{i})>0$ whenever $%
n\geq 11$. Hence we have $n-1$ vertices with positive leverage centrality.
\end{proof}

We present a second example. Let $G$ be a graph with $n\geq 12$ vertices $%
v_{1},v_{2},...,v_{n}$ and edges: $\{v_{i}v_{j}$ $|$ $1\leq i<j\leq n-5\}$ $%
\cup $ $\{v_{i}v_{j}$ $|$ $1\leq i\leq n-5$ and $n-4\leq j\leq n-1\}$ $\cup $
$\left\{ v_{n-4}v_{n-2}\right\} $ $\cup $ $\left\{ v_{n-3}v_{n-1}\right\} $ $%
\cup $ $\{v_{i}v_{n}$ $|$ $n-3\leq i\leq n-1\}$. It is clear that $%
l(v_{i})>0 $ for all $1\leq i\leq n-5$ since these vertices have the maximum
degree. Then for $n-4\leq i\leq n-1$, $l(v_{i})=\frac{1}{n-3}\left( \left(
n-5\right) \frac{(n-3)-(n-2)}{(n-3)+(n-2)}+\frac{(n-3)-4}{(n-3)+4}\right)
=\allowbreak \frac{n^{2}-15n+40}{2n^{3}-9n^{2}+4n+15}$ which is positive
when $n>11.531$.

\subsection{Leverage Centrality vs. Degree Centrality}

Degree centrality weights a vertex based on its degree. A vertex with higher
(lower) degree is deemed more (less) central. This property has been
well-studied (for early works see Czepiel \cite{Czepiel}, Faucheaux and
Moscovici \cite{Faucheux}, Freeman \cite{Freeman}, Garrison, \cite{Garrison}%
, Hanneman and Newman \cite{Hanneman}, Kajitani and Maruyama \cite{Kajitani}%
, Mackenzie \cite{Mackenzie}, Nieminen \cite{Nieminen1}, \cite{Nieminen2},
Pitts \cite{Pitts}, Rogers \cite{Rogers}, and Shaw \cite{Shaw}). For some
families of graphs the leverage centrality and degree centralities of
vertices are closely related. For example, in scale-free networks where the
distribution of degrees follows the power law, vertices with large degree
will be adjacent to many vertices with much lower degrees. Hence the
leverage centrality of these vertices will also be high.

However, for other families of graphs leverage centrality and degree
centrality are not closely related. We show in the following example it is
possible to construct infinite families of graphs where the vertex of
largest degree does not have the highest leverage centrality. We do this by
connecting nearly complete graphs as shown in Figure 3.

\FRAME{dtbpFU}{3.7421in}{1.2358in}{0pt}{\Qcb{Figure 3. A family of connected
nearly complete graphs}}{}{Figure}{\special{language "Scientific Word";type
"GRAPHIC";maintain-aspect-ratio TRUE;display "USEDEF";valid_file "T";width
3.7421in;height 1.2358in;depth 0pt;original-width 3.6945in;original-height
1.2012in;cropleft "0";croptop "1";cropright "1";cropbottom "0";tempfilename
'OJXR2T0Q.wmf';tempfile-properties "XPR";}}

For all $n\geq 5$, we have $\deg (u)>\deg (v)$, however $l(u)<l(v)$.

Let $u$ be a vertex in $K_{n+1}$ that has a neighbor vertex on the $K_{n}$
graph. Then, $\deg (u)=n$ and as $n\rightarrow \infty $, it follows that $%
\deg (u)\rightarrow \infty $. Let $v$ be the vertex that is the base of the
claw graph found on the right side of the graph shown in Figure 2. Thus, the
degree of $v$ will always equal $4$ and therefore, for all $n\geq 5$, $\deg
(u)>\deg (v)$.

Since we know the degree of the neighbors of $u$, we can calculate the
leverage centrality of $u$ as shown: 
\begin{equation*}
l(u)=\frac{1}{n}\left( \frac{n-(n-1)}{n+(n-1)}+(n-1)\left( \frac{n-n}{n+n}%
\right) \right) =\frac{1}{2n^{2}-n}\text{.}
\end{equation*}%
Thus, if we take the limit of the leverage centrality of $u$ as $%
n\rightarrow \infty $ we get: 
\begin{equation*}
\lim_{n\rightarrow \infty }\frac{1}{2n^{2}-n}=0\text{.}
\end{equation*}%
We can also calculate the leverage centrality of $v$: 
\begin{equation*}
l(v)=\frac{1}{4}\left( \frac{4-2}{4+2}+(3)\left( \frac{4-1}{4+1}\right)
\right) =\frac{8}{15}\text{.}
\end{equation*}%
Since the leverage centrality of $u$ converges to $0$ as $n\rightarrow
\infty $, and the leverage centrality of $v$ is equal to $\frac{8}{15}$,
then $l(v)>l(u)$ $\forall n\geq 5$.

\subsection{Leverage Centrality Zero}

We note that bounds given in Lemma \ref{R1} are tight for regular graphs,
where the leverage centrality of all vertices is zero. In fact, it is
straightforward to show that $l(v)=0$ for every vertex $v$ if and only if $G$
is a regular graph. It is also clear that for a vertex $v$ with degree $k$
that if all of the neighbors of $v$ have degree $k$, then $l(v)=0$. However,
it is possible for a vertex to have a leverage centrality of zero without
all of its neighbors having the same degree as the original vertex. We
investigate this property below.

\begin{example}
\label{1}Let $G$ be a graph containing a vertex $v$ of degree $k$ where $k-1$
of $v$\textquotedblright s neighbors have degree $k=2$ and the remaining
neighbor has degree $1$. Then $l(v)=\frac{1}{k}\left( \frac{k-1}{k+1}%
+(k-1)\left( \frac{k-(k+2)}{k+(k+2)}\right) \right) \medskip =0$.
\end{example}

We also give an example of a graph with a vertex $v$ whose neighbors all
have distinct degrees and $l(v)=0$.

\begin{example}
\label{2}Let $G$ be a graph containing a vertex $v$ of degree $3$ and the
neighbors of $v$ have degrees $1$, $2$, and $17$. The leverage centrality of 
$v$ is $l(v)=\frac{1}{3}\left( \frac{3-1}{3+1}+\frac{3-2}{3+2}+\frac{3-17}{%
3+17}\right) =0$.
\end{example}

It would be an interesting problem indeed to determine necessary and
sufficient conditions for a vertex $v$ to have leverage centrality zero,
particularly when the neighbors of $v$ all have distinct degrees. A computer
search gives several examples for vertices with small degree.\medskip

\begin{center}
$%
\begin{tabular}{|l|l|}
\hline
$d(v)$ & degrees of the neighbors of $v$ \\ \hline
$3$ & $1,2,17$ \\ \hline
$3$ & $1,3,9$ \\ \hline
$4$ & $1,2,5,41$ \\ \hline
$5$ & $1,2,4,13,37$ \\ \hline
$5$ & $1,2,5,10,37$ \\ \hline
$5$ & $1,3,5,7,35$ \\ \hline
$6$ & $1,2,3,6,36,66$ \\ \hline
\end{tabular}%
$ $\ 
\begin{tabular}{|l|l|}
\hline
$d(v)$ & degrees of the neighbors of $v$ \\ \hline
$7$ & $1,2,3,7,11,33,77$ \\ \hline
$7$ & $1,2,3,9,11,33,41$ \\ \hline
$7$ & $1,2,3,11,13,23,33$ \\ \hline
$7$ & $1,2,5,7,11,21,49$ \\ \hline
$7$ & $1,2,5,7,11,28,33$ \\ \hline
$7$ & $1,2,5,8,13,17,38$ \\ \hline
$7$ & $1,2,5,9,11,13,73$ \\ \hline
$7$ & $1,2,5,11,14,17,21$ \\ \hline
$7$ & $1,3,4,5,8,37,81$ \\ \hline
$7$ & $1,3,5,7,8,21,49$ \\ \hline
$7$ & $1,3,5,7,8,28,33$ \\ \hline
$7$ & $1,3,5,8,9,13,73$ \\ \hline
$7$ & $1,3,5,8,14,17,21$ \\ \hline
\end{tabular}%
$
\end{center}

\section{Complete Multipartite Graphs}

We use $K_{t_{1},t_{2},\dots ,t_{r}}$ to denote the complete multipartite
graph with parts of sizes $t_{1},t_{2},..,t_{r}$ and each vertex in a part
is adjacent to every vertex in each of the other parts. As noted in \cite%
{LC2} for vertices in the star graph $K_{1,n-1}$ the leverage centrality
meets the two extremes. The vertex in a part by itself has leverage
centrality $\frac{1}{n-1}\left( (n-1)\frac{\left( n-1\right) -1}{(n-1)+1}%
\right) =1-\frac{2}{n}$ and all other vertices have a leverage centrality of 
$\frac{1}{1}\left( \frac{1-\left( n-1\right) }{1+(n-1)}\right) =-1+\frac{2}{n%
}$.

We can extend the same idea to the general case of complete multipartite
graphs. We will use $G=K_{t_{1},t_{2},\dots ,t_{r}}$ to denote a complete
multipartite graph with $r$ parts $n_{1},n_{2},...,n_{r}$ where each part $%
n_{i}$ has order $t_{i}$ for all $1\leq i\leq r$.

\begin{theorem}
Let $G=K_{t_{1},t_{2},\dots ,t_{r}}$ where $t_{i}$ is the order of part $%
n_{i}$. Then 
\begin{equation*}
l(v_{i})=\frac{1}{\sum_{j\neq i}t_{j}}\left( \sum_{k\neq i}t_{k}\left( \frac{%
t_{k}-t_{i}}{\sum_{j\neq i}t_{j}+\sum_{j\neq k}t_{j}}\right) \right)
\end{equation*}
\end{theorem}

\begin{proof}
Let $v_{i}$ be a vertex in part $n_{i}$ with degree $\sum_{j\neq i}t_{j}$.
Due to the nature of a complete multipartite graph, it follows that $v_{i}$
will have $t_{1}$ neighbors in part $n_{1}$, $t_{2}$ neighbors in part $%
n_{2} $, $t_{i}$ neighbors in part $n_{i}$, and the pattern continues for
all $1\leq i\leq r$ groups. Note that every vertex $v_{k}\in n_{k}$ will
have degree $\sum_{j\neq k}t_{k}$. Thus the leverage centrality of $v_{i}$
can be calculated as follows: 
\begin{equation*}
l(v_{i})=\frac{1}{\sum_{j\neq i}t_{j}}\left( t_{1}\left( \frac{\sum_{j\neq
i}t_{j}-\sum_{j\neq 1}t_{j}}{\sum_{j\neq i}t_{j}+\sum_{j\neq 1}t_{j}}\right)
+t_{2}\left( \frac{\sum_{j\neq i}t_{j}-\sum_{j\neq 2}t_{j}}{\sum_{j\neq
i}t_{j}+\sum_{j\neq 2}t_{j}}\right) +\cdots +t_{r}\left( \frac{\sum_{j\neq
i}t_{j}-\sum_{j\neq r}t_{j}}{\sum_{j\neq i}t_{j}+\sum_{j\neq r}t_{j}}\right)
\right)
\end{equation*}%
\begin{equation*}
=\frac{1}{\sum_{j\neq i}t_{j}}\left( \sum_{k\neq i}t_{k}\left( \frac{%
\sum_{j\neq i}t_{j}-\sum_{j\neq k}t_{j}}{\sum_{j\neq i}t_{j}+\sum_{j\neq
k}t_{j}}\right) \right)
\end{equation*}%
\begin{equation*}
=\frac{1}{\sum_{j\neq i}t_{j}}\left( \sum_{k\neq i}t_{k}\left( \frac{%
t_{k}-t_{i}}{\sum_{j\neq i}t_{j}+\sum_{j\neq k}t_{j}}\right) \right)
\end{equation*}

This completes the proof.
\end{proof}

\section{Cartesian Product of Graphs}

\begin{definition}
Given a graph $F$ with vertex set $V(F)$ and edge set $E(F)$, and a graph $H$
with vertex set $V(H)$ and edge set $E(H)$ we let $G$ define the Cartesian
Product of $F$ and $H$ to be the graph $G=F\times H$ which is defined as
follows:\ $V(G)=\{(u,v)|u\in V(F)$ and $v\in V(H)\}$ and $%
E(G)=\{(u_{1},v_{1}),(u_{2},v_{2})$ where $u_{1}=u_{2}$ and $%
(v_{1},v_{2})\in E(H)$ or $v_{1}=v_{2}$ and $(u_{1},u_{2})\in E(F)\}$. We
use $\underset{m}{{\LARGE \times }}G_{i}$ to denote the Cartesian product of 
$m$ copies of a graph $G_{i}$.
\end{definition}

We next present an elementary result from graph theory.

\begin{lemma}
\label{degree sum}If $G=F\times H$, then the degree of a vertex $(u,v)$ in $%
G $ is the sum of the degrees of vertices $u$ and $v$, where $u\in V(F)$ and 
$v\in V(H)$.
\end{lemma}

\begin{theorem}
Let $G$ be a graph and let $G_{r}$ be a regular graph where each vertex has
degree $r$. Let $u\in V(G_{r})$ and let $v_{i}$ and $v_{j}$ be vertices in $%
G $ with degrees $k_{i}$ and $k_{j}$ respectively. For each vertex $%
(u,v_{i})\in V\left( G_{r}\times G\right) $ we have
\end{theorem}

\begin{center}
$l(u,v_{i})=\frac{1}{r+k_{i}}\dsum\limits_{j\neq i}\frac{k_{i}-k_{j}}{%
2r+k_{i}+k_{j}}$.
\end{center}

\begin{proof}
Consider a vertex $(u,v_{i})\in V\left( G_{r}\times G\right) $. We note that 
$\deg ((u,v_{i}))=\deg (u)+\deg (v_{i})=r+k_{i}$.\medskip\ Then

\qquad \qquad \qquad \qquad \qquad\ $l(u,v_{i})=\frac{1}{r+k_{i}}%
\dsum\limits_{j\neq i}\frac{\left( r+k_{i}\right) -\left( r+k_{j}\right) }{%
2r+k_{i}+k_{j}}=\frac{1}{r+k_{i}}\dsum\limits_{j\neq i}\frac{k_{i}-k_{j}}{%
2r+k_{i}+k_{j}}$.
\end{proof}

\begin{corollary}
Let $(u,v_{i})$ be a vertex in $K_{m}\times G$ where $u\in V\left(
K_{m}\right) $ and $v_{i}\in V(G)$. Then for all $(v_{i},v_{j})\in E(G)$
\end{corollary}

\begin{center}
$l(u,v_{i})=\frac{1}{(m-1)+\deg (v_{i})}$ $\dsum\limits_{j}\frac{(m-1)+\deg
(v_{i})-\left( (m-1)+\deg (v_{j})\right) }{(m-1)+\deg (v_{i})+\left(
(m-1)+\deg (v_{j})\right) }$

$=\frac{1}{(m-1)+\deg (v_{i})}$ $\dsum\limits_{j}\frac{\deg (v_{i})-\deg
(v_{j})}{(2m-2)+\deg (v_{i})+\deg (v_{j})}$.
\end{center}

\begin{proof}
By Lemma \ref{degree sum} we have that $\deg \left( (u,v_{i})\right)
=m-1+\deg (v_{i})$ and for all neighbors $v_{j}$ of vertex $v_{i}$ we have
that $\deg \left( (u,v_{j})\right) =m-1+\deg (v_{j})$. The result then
follows.
\end{proof}

\subsection{Cartesian Products of $P_{n}$}

In this section we will consider the lattice, $\underset{m}{{\LARGE \times }}%
P_{n}$. As the calculation of the degrees of vertices in a lattice is
straightforward we will present results only involving the degrees without
proof. We continue with some definitions.

\begin{definition}
Any vertex of $\underset{m}{{\LARGE \times }}P_{n}$ can be defined by an $m$%
-tuple: 
\begin{equation*}
v=(v_{1},v_{2},\cdots ,v_{m})\quad \text{such that }v_{i}\in \{1,\dots
,n\}\quad \forall i\in \{1,\dots ,m\}\text{.}
\end{equation*}
\end{definition}

\begin{definition}
We define a \textbf{corner vertex} of $\underset{m}{{\LARGE \times }}P_{n}$
to be 
\begin{equation*}
v_{c}=(v_{1},v_{2},\dots ,v_{m})\quad \text{such that }v_{i}\in \{1,n\}\quad
\forall i\in \{1,\dots ,m\}\text{.}
\end{equation*}

A \textbf{non-corner vertex} is a vertex $v=(v_{1},v_{2},\dots ,v_{m})$ of $%
\underset{m}{{\LARGE \times }}P_{n}$ such that at least one $v_{i}\in
\{2,\dots ,n-1\}$.
\end{definition}

An \textbf{inner corner vertex} of $\underset{m}{{\LARGE \times }}P_{n}$ is
defined as follows. 
\begin{equation*}
v_{ic}=(v_{1},v_{2},\dots ,v_{m})\quad \text{such that }v_{i}\in
\{2,n-1\}\quad \forall i\in \{1,\dots ,m\}\text{.}
\end{equation*}%
It follows by definition that all vertices that are \textit{inner corner}
vertices are also \textit{non-corner} vertices.

We note that

\begin{equation*}
\deg (v)=\sum_{i=1}^{m}x_{i}\quad \text{such that }x_{i}=%
\begin{cases}
1\text{ if }v_{i}\in \{1,n\} \\ 
2\text{ if }v_{i}\in \{2,\dots ,n-1\}%
\end{cases}%
\end{equation*}%
We also observe that neighbor $v^{\prime }$ of vertex $v=(v_{1},v_{2},\dots
,v_{m})$ is defined as $v^{\prime }=(v_{1}^{^{\prime }},v_{2}^{^{\prime
}},\dots ,v_{m}^{^{\prime }})$ such that $v_{i}^{^{\prime }}=v_{i}$ for $m-1$
elements of $(v_{1}^{^{\prime }},v_{2}^{^{\prime }},\dots ,v_{m}^{^{\prime
}})$ and 
\begin{equation*}
\left\vert v_{i}^{^{\prime }}-v_{i}\right\vert =1
\end{equation*}%
for the remaining element of the $m$-tuple $(v_{1}^{^{\prime
}},v_{2}^{^{\prime }},\dots ,v_{m}^{^{\prime }})$. Notice, there are two
special cases for this remaining element. If the remaining element $v_{i}=1$%
, then $v_{i}^{^{\prime }}=2$ and if the remaining element $v_{i}=n$, then $%
v_{i}^{^{\prime }}=n-1$.

\subsubsection{General Lemmas}

We begin with a basic result involving the degrees of vertices and its
neighbors in a lattice.

\begin{lemma}
\label{lemma: adjacent degrees} Let $G$ be a lattice $\underset{m}{{\LARGE %
\times }}P_{n}$. Any vertex adjacent to a vertex with degree $k$ must have
degree $k-1$, $k$, or $k+1$.
\end{lemma}

\subsubsection{Extreme Leverage Centralities}

We next identify vertices with the minimum and maximum leverage
centralities. We will show that the vertices with the minimum leverage
centrality are the corners and the vertices with the maximum leverage
centrality are the inner corners. Furthermore, we will show that for any
vertex $v$ in the lattice $G$ $=$ $\underset{m}{{\LARGE \times }}P_{n}$, $-%
\frac{1}{2m+1}\leq l(v)\leq \frac{1}{8m-2}$.

\paragraph{Minimum Leverage Centrality}

We first characterize the vertices with the minimum leverage centrality. We
begin by stating two elementary lemmas involving degrees of vertices in a
lattice.

\begin{lemma}
\label{16}Any corner vertex $v_{c}$ in $G=$ $\underset{m}{{\LARGE \times }}%
P_{n}$ will have a degree of $m$. Furthermore, each neighbor of $v_{c}$ will
have degree of $m+1$.
\end{lemma}

\begin{lemma}
\label{corner_adjacent} Let $G$ be the lattice $\underset{m}{{\LARGE \times }%
}P_{n}$. A vertex $v$ that is non-corner vertex of $G$ must have at least
one neighbor $u$ such that: 
\begin{equation*}
\deg (u)\leq \deg (v)\text{.}
\end{equation*}
\end{lemma}

\begin{theorem}
\label{corner_lev_cen} Let $v_{c}$ be a corner vertex of $G=$ $\underset{m}{%
{\LARGE \times }}P_{n}$. Then, 
\begin{equation*}
\text{ }l(v_{c})=-\frac{1}{2m+1}\text{.}
\end{equation*}
\end{theorem}

\begin{proof}
By Lemma \ref{16} we have that for $G$ $=$ $\underset{m}{{\LARGE \times }}%
P_{n}$, $\deg (v_{c})=m$ and that for a neighbor $u$ of $v_{c}$, $\deg
(u)=m+1$. We can compute the leverage centrality of $v_{c}$ with Definition %
\ref{def: leverage centrality}.

\begin{equation*}
l(v_{c})=\frac{1}{m}\sum_{i=1}^{m}\frac{m-(m+1)}{m+(m+1)}=\frac{1}{m}%
\sum_{i=1}^{m}\frac{-1}{2m+1}=\frac{1}{m}\cdot m\left( \frac{-1}{2m+1}%
\right) =-\frac{1}{2m+1}\text{.}
\end{equation*}
\end{proof}

\begin{theorem}
(Minimum Leverage Centrality) \label{theorem: minimum levarage centrality}
Let $u$ be any vertex in $G=$ $\underset{m}{{\LARGE \times }}P_{n}$ that is
not a corner vertex and let $v_{c}$ be a corner vertex in $G$. Then, $%
l(v_{c})<l(u)$.
\end{theorem}

\begin{proof}
Let $v$ be a non-corner vertex in $G$ with degree $k$. We know from Lemma %
\ref{corner_adjacent} that at least one adjacent node has degree at most $k$%
. We know from Lemma \ref{lemma: adjacent degrees} that the remaining
adjacent nodes can have degree at most $k+1$.

Let $v$ have one adjacent node with degree $k$ and $k-1$ adjacent nodes with
degree $k+1$. We now calculate the leverage centrality of $v$. 
\begin{equation*}
l(v)=\frac{1}{k}\left( \frac{k-(k+1)}{k+(k+1)}\cdot (k-1)+\frac{k-k}{k+k}%
\right) =\left( \frac{1-k}{k(2k+1)}\right) \text{.}
\end{equation*}

From Theorem \ref{corner_lev_cen}, we have that for a corner vertex $v_{c}$
of degree $k$, the leverage centrality is: 
\begin{equation*}
l(v_{c})=\left( -\frac{1}{2k+1}\right) \text{.}
\end{equation*}%
Given that the degree of any adjacent node must be greater than $0$, we know
that $0\leq \frac{k-1}{k}<1$. It follows that $\left( -\frac{1}{2k+1}\right)
<\frac{1-k}{k(2k+1)}$ and hence $l(v_{c})<l(v)$.

If the neighbors of any non-corner vertex $u$ differ from that of $v$, then
it follows from our construction of $v$ and Lemma \ref{lemma: adjacent
degrees} that for any corresponding neighbors $u_{i}$ from $u$ and $v_{i}$
from $v$, that $\deg (u_{i})\leq \deg (v_{i})$ and hence, $l(v)\leq l(u)$.
So we have that $l(v_{c})<l(v)\leq l(u)$.

This implies that $l(v_{c})<l(u)$ which completes the proof.
\end{proof}

\paragraph{Maximum Leverage Centrality}

We next characterize the vertices with the largest leverage centrality,
beginning with two elementary results involving degrees of vertices in a
lattice.

\begin{lemma}
\label{lemma: neighbours of inner corner} Let $v_{ic}$ be an inner vertex of 
$G=\underset{m}{\text{ }{\LARGE \times }}P_{n}$. Then $v_{ic}$ has $2m$
neighbors, such that $m$ neighbors have degree $2m$ and the remaining $m$
neighbors have degree $2m-1$.
\end{lemma}

\begin{lemma}
\label{lemma: maximum degrees for a vertex} Let $v$ be a vertex in $G$ $=$ $%
\underset{m}{{\LARGE \times }}P_{n}$. Then, $\deg (v)\leq 2m$.
\end{lemma}

\begin{theorem}
(Maximum leverage centrality) \label{theorem: maximum leverage centrality}
Let $u$ be a vertex in $G=$ $\underset{m}{{\LARGE \times }}P_{n}$ that is
not an inner corner vertex of $G$, and let $v_{ic}$ be an inner corner
vertex in $G$. Then, $l(u)<l(v_{ic})$. Furthermore, $l(v_{ic})=\frac{1}{8m-2}
$.
\end{theorem}

\begin{proof}
Let $v_{ic}$ be an inner corner vertex of $G$. We have that 
\begin{equation*}
v_{ic}=(v_{1},v_{2},\dots ,v_{m})\quad \text{such that }v_{i}\in
\{2,n-1\}\quad \forall i\in \{1,\dots ,m\}\text{.}
\end{equation*}%
By Lemma \ref{lemma: neighbours of inner corner}, we know that $\deg
(v_{ic})=2m$. We are also given that $m$ neighbors of $v_{ic}$ have degree $%
2m$ and that $m$ neighbors of $v_{ic}$ have degree $2m-1$. The leverage
centrality of $v_{ic}$ is 
\begin{equation*}
l(v_{ic})=\frac{1}{2m}\left[ \left( \underbrace{\frac{2m-2m}{2m+2m}+\cdots +%
\frac{2m-2m}{2m+2m}}_{m\text{ terms}}\right) +\left( \underbrace{\frac{%
2m-(2m-1)}{2m+(2m-1)}+\cdots +\frac{2m-(2m-1)}{2m+(2m-1)}}_{m\text{ terms}%
}\right) \right] \text{.}
\end{equation*}%
\noindent By rearranging terms we get: 
\begin{equation}
\frac{1}{2m}\left[ \underbrace{\left( \frac{2m-2m}{2m+2m}+\frac{2m-(2m-1)}{%
2m+(2m-1)}\right) +\cdots +\left( \frac{2m-2m}{2m+2m}+\frac{2m-(2m-1)}{%
2m+(2m-1)}\right) }_{m\text{ terms}}\right] \text{.}
\label{eqn: terms of lc for inner vertex}
\end{equation}%
By distributing $\frac{1}{2m}$ we get that each term of the sum for $%
l(v_{ic})$ can be expressed as: 
\begin{equation*}
\frac{1}{2m}\left[ \frac{2m-2m}{2m+2m}+\frac{2m-(2m-1)}{2m+(2m-1)}\right]
\end{equation*}%
and since there are $m$ terms in the sum, we can express $l(v_{ic})$ as: 
\begin{equation}
l(v_{ic})=m\cdot \frac{1}{2m}\left[ \frac{2m-2m}{2m+2m}+\frac{2m-(2m-1)}{%
2m+(2m-1)}\right] \text{.}
\end{equation}%
We simplify this to get: 
\begin{equation*}
l(v_{ic})=m\cdot \frac{1}{2m}\left[ \frac{2m-2m}{2m+2m}+\frac{2m-(2m-1)}{%
2m+(2m-1)}\right] =\frac{1}{2}\left[ \frac{2m-2m}{2m+2m}+\frac{2m-(2m-1)}{%
2m+(2m-1)}\right] =\frac{1}{8m-2}\text{,}
\end{equation*}%
which proves the second part of the theorem.

Let $u$ be a vertex in $G$ that is not an inner corner vertex of $G$. We
have that $\exists u_{i}^{\ast }\in u=(u_{1},u_{2},\dots ,u_{m})$ such that $%
u_{i}^{\ast }\in \{1,3,\dots ,n-2,n\}$. Without loss of generality, we can
assume that $u_{i}=v_{i}$ when $u_{i}\not=u_{i}^{\ast }$ and thus $u$ and $%
v_{ic}$ differ only in one element, $u_{i}^{\ast }\in u$ and $v_{i}^{\ast
}\in v_{ic}$ where $u_{i}^{\ast }\not=v_{i}^{\ast }$.

We see that two cases arise in calculating the leverage centrality of $u$.

\begin{enumerate}
\item[(i)] Let $u_{i}^{\ast }\in \{1,n\}$ and $v_{i}^{\ast }\in \{2,n-1\}$

\noindent By Lemma \ref{lemma: neighbours of inner corner}, we have that $%
\deg (u)=2m-1$. In calculating the leverage centrality of $u$, we see that $%
l(u)$ and $l(v_{ic})$ can differ only in one term of Equation \ref{eqn:
terms of lc for inner vertex} such that: 
\begin{align*}
l(u)& =(m-1)\cdot \frac{1}{2m}\left[ \frac{2m-2m}{2m+2m}+\frac{2m-(2m-1)}{%
2m+(2m-1)}\right] +\frac{1}{2m-1}\left[ \frac{(2m-1)-2m}{(2m-1)+2m}\right] \\
l(v_{ic})& =(m-1)\cdot \frac{1}{2m}\left[ \frac{2m-2m}{2m+2m}+\frac{2m-(2m-1)%
}{2m+(2m-1)}\right] +\frac{1}{2m}\left[ \frac{2m-2m}{2m+2m}+\frac{2m-(2m-1)}{%
2m+(2m-1)}\right] \\
\text{Let }q& =(m-1)\cdot \frac{1}{2m}\left[ \frac{2m-2m}{2m+2m}+\frac{%
2m-(2m-1)}{2m+(2m-1)}\right] \text{.}
\end{align*}%
\begin{align*}
\text{Then }l(u)& =q+\frac{1}{2m-1}\left[ \frac{(2m-1)-2m}{(2m-1)+2m}\right]
=q-\frac{1}{2m-1}\left[ \frac{1}{4m-1}\right] \\
\text{ and }l(v_{ic})& =q+\frac{1}{2m}\left[ \frac{2m-2m}{2m+2m}+\frac{%
2m-(2m-1)}{2m+(2m-1)}\right] =\frac{2-m}{2m(1-4m)}\text{.}
\end{align*}%
For the differing terms for the expressions for leverage centrality of $u$
and $v_{ic}$ we see that 
\begin{equation*}
-\frac{1}{2m-1}\left[ \frac{1}{4m-1}\right] <\frac{1}{2m}\left[ \frac{1}{4m-1%
}\right] \text{.}
\end{equation*}%
and it follows that 
\begin{equation*}
l(u)<l(v_{ic})\text{.}
\end{equation*}

\item[(ii)] If $u_{i}^{\ast }\in \{3,n-2\}$ and $v_{i}^{\ast }\in \{2,n-1\}$

\noindent By Lemma \ref{lemma: neighbours of inner corner}, we know that $%
\deg (u)=2m$. In calculating the leverage centrality of $u$, we see that $%
l(u)$ and $l(v_{ic})$ can differ only in one term of Equation \ref{eqn:
terms of lc for inner vertex} such that: 
\begin{align*}
l(u)& =(m-1)\cdot \frac{1}{2m}\left[ \frac{2m-2m}{2m+2m}+\frac{2m-(2m-1)}{%
2m+(2m-1)}\right] +\frac{1}{2m}\left[ \frac{2m-2m}{2m+2m}\right] \\
l(v_{ic})& =(m-1)\cdot \frac{1}{2m}\left[ \frac{2m-2m}{2m+2m}+\frac{2m-(2m-1)%
}{2m+(2m-1)}\right] +\frac{1}{2m}\left[ \frac{2m-2m}{2m+2m}+\frac{2m-(2m-1)}{%
2m+(2m-1)}\right] \\
\text{Let }q& =(m-1)\cdot \frac{1}{2m}\left[ \frac{2m-2m}{2m+2m}+\frac{%
2m-(2m-1)}{2m+(2m-1)}\right] =\frac{m-1}{2m(4m-1)}
\end{align*}%
From the proof of Case (i), we already have $l(v_{ic})$ 
\begin{align*}
l(u)& =q+\frac{1}{2m}\left[ \frac{2m-2m}{2m+2m}\right] =q\text{, and} \\
l(v_{ic})& =q+\frac{1}{2m}\left[ \frac{1}{4m-1}\right] \text{.}
\end{align*}%
For the differing terms for the expressions for leverage centrality of $u$
and $v_{ic}$ we see that 
\begin{equation*}
0<\frac{1}{2m}\left[ \frac{1}{4m-1}\right]
\end{equation*}%
and it follows that 
\begin{equation*}
l(u)<l(v_{ic})\text{.}
\end{equation*}
\end{enumerate}

In both Cases (i) and (ii), we find that $l(u)<l(v_{ic})$ which proves that
first part of the theorem and completes the proof.
\end{proof}

\subsubsection{Convergence of Leverage Centrality as $m\rightarrow \infty $}

We next consider the leverage centrality of different vertices as the number
of dimensions is increased.

\begin{theorem}
\label{theorem: convergence of leverage centralities} As the number of paths
in the Cartesian product increases $\left( m\rightarrow \infty \right) $,
the leverage centralities of all of the vertices of $G=$ $\underset{m}{%
{\LARGE \times }}P_{n}$ converge to $0$.
\end{theorem}

\begin{proof}
Let $G=$ $\underset{m}{{\LARGE \times }}P_{n}$. From Theorem \ref{theorem:
maximum leverage centrality}, we know that for any $m$, the maximum leverage
centrality of any vertex $v$ of $G$ is: 
\begin{equation*}
\max (l(v))=\frac{1}{8m-2}
\end{equation*}%
From Theorem \ref{theorem: minimum levarage centrality}, we know that for
any $m$, the minimum leverage centrality of any vertex $v$ of $G$ is: 
\begin{equation*}
\min (l(v))=-\frac{1}{2m+1}\text{.}
\end{equation*}%
Therefore, for any vertex $v$ in $G$ the leverage centrality is bounded as
follows: 
\begin{equation*}
-\frac{1}{2m+1}\leq l(v)\leq \frac{1}{8m-2}\text{.}
\end{equation*}%
We see that 
\begin{equation*}
\lim_{m\rightarrow \infty }\left( -\frac{1}{2m+1}\right) =\lim_{m\rightarrow
\infty }\left( \frac{1}{8m-2}\right) =0\text{.}
\end{equation*}%
It follows that 
\begin{equation*}
\lim_{m\rightarrow \infty }l(v)=0\text{.}
\end{equation*}%
which completes the proof.
\end{proof}

\section{Leverage Centralities in Lattices and Triangle Numbers}

In this section we investigate the number of distinct leverage centralities
for lattices and show there is a surprising connection to the triangle
numbers $\binom{m+2}{2}$ where $m\geq 1$. We can label the vertices of $%
\underset{m}{{\LARGE \times }}P_{n}$ with using $m$-tuples where $%
v=(v_{1},v_{2},\cdots ,v_{m})$ such that $v_{i}\in \{1,\dots ,n\}\quad
\forall i\in \{1,\dots ,m\}$. For simplicity we will denote $v_{r,s,t}$ by $%
(r,s,t)$.

\begin{itemize}
\item There are three distinct leverage centralities for $P_{n}$ where $%
n\geq 5$. Let $V(P_{n})=v_{1},v_{2},...,v_{n}$ where $n\geq 5$. Then $%
l(v_{1})=l(v_{n})=-\frac{1}{3}$; $l(v_{2})=l(v_{n-1})=\frac{1}{6}$; and $%
l(v_{i})=0$ for all other $v_{i}$.

\item For $P_{n}\times P_{n}$ where $n\geq 5$, we have six different
leverage centralities:
\end{itemize}

$\qquad l(1,1)=\frac{-1}{5}$, $l(1,2)=\frac{-1}{5}$, $l(1,3)=\frac{-1}{5}$, $%
l(2,2)=\frac{-1}{5}$, $l(2,3)=\frac{-1}{5}$, and $l(3,3)=0$.

\begin{itemize}
\item For $P_{n}\times P_{n}\times P_{n}$ where $n\geq 5$, we have ten
different leverage centralities:
\end{itemize}

$\qquad l(1,1,1)=\frac{-1}{7}$, $l(1,1,2)=\frac{-5}{252}$, $l(1,1,3)=\frac{-1%
}{18}$, $l(1,2,2)=\frac{13}{495}$, $l(1,2,3)=\frac{2}{495}$,\medskip

$\qquad l(1,3,3)=\frac{-1}{55}$, $l(2,2,2)=\frac{1}{22}$, $l(2,2,3)=\frac{1}{%
30}$, $l(2,3,3)=\frac{1}{33}$, and $l(3,3,3)=0$.\medskip

By symmetry we need only consider vertices with coordinates $1\leq v_{i}\leq
3$ and $v_{i}\leq v_{i+1}$ for all $1\leq i\leq m-1$. It is straightforward
to count the number of different combinations of a degree of a vertex and
the degrees of its neighbors. We need only count the number of solutions to
the equation $x_{1}+x_{2}+x_{3}=m$ where $x_{i}$ is the number of times $i$
appears in the coordinate. This can be done using the following lemma.

We next restate a well-known combinatorial formula.

\begin{lemma}
\label{cn}The number of solutions to $x_{1}+x_{2}+\cdots +x_{n}=m$ where
each $x_{i}\in \mathbb{N}$ is $\binom{n+m-1}{m-1}$.
\end{lemma}

Using Lemma \ref{cn}, the number of solutions to this equation is the $(m+1)$%
-st triangle number, $\binom{m+2}{2}$.

Hence we have the following upper bound.

\begin{theorem}
\label{main}If $n\geq 5$ the number of distinct leverage centralities in $G=$
$\underset{m}{{\LARGE \times }}P_{n}$ is less than or equal to$\binom{m+2}{2}
$.
\end{theorem}

For small cases of $m$ this bound is in fact tight. The first three cases
have been shown above. In the next theorem we show that this holds for $m<7$.

\begin{theorem}
Let $k={\binom{m+2}{2}}$ and $G=P_{k_{1}}\times P_{k_{2}}\times \cdots
\times P_{k_{m}}$ where $k_{1}=k_{2}=\cdots =k_{m}\geq 5$ with vertices $%
V=\{v_{0},v_{2},...,v_{k-1}\}$.

\begin{enumerate}
\item If $t_{j}$ is the $j$th triangular number for $0\leq j\leq m$ and $%
r=t_{j}+i$ where $0\leq i\leq j$, then leverage centrality of $v_{r}$ is
given by 
\begin{equation*}
l(v_{r})=\frac{1}{m+j}\left[ \frac{j-i}{2(m+j)-1}-\frac{(m-j)}{2(m+j)+1}%
\right] \text{.}
\end{equation*}

\item The number of distinct leverage centralities in $G$ is less than or
equal to ${\binom{m{+2}}{{2}}}$. Moreover, if $m<7$ the equality holds.
\end{enumerate}
\end{theorem}

\begin{proof}
We first prove Property 1. Let $v_{r}$ be $r$-th $n$-tuple that appears in
the lexicographical ordering where each term is between $1$ and $3$
inclusive, i.e., 
\begin{equation*}
v_{1}=(1,1,1,...,1,1),v_{2}=(1,1,1,...,1,2),v_{3}=(1,1,1,...,1,3),...,v_{k}=(3,3,3,...,3,3).
\end{equation*}

From this set of vertices $V=\{v_{1},v_{2},...,v_{k}\}$ we can see that the
degree of each vertex $v_{r}$ is $m+j$ where $r=$ $t_{j}+i$ and $t_{j}$ is
the $j$th triangular number. The degrees of the vertices adjacent to $v_{r}$
are as follows: $m-j$ vertices of degree $m+j+1$, there are $j-i$ vertices
of degree $m+j-1$ and there are $j+i$ vertices of degree $m+j$. Therefore,
for $0\leq j\leq m$ and $0\leq i\leq j$ the leverage centrality for each
vertex $v_{r}$ is:

\begin{equation*}
l(v_{r})=\frac{1}{m+j}\left[ \frac{j-i}{2m+(2j-1)}-\frac{(m-j)}{2m+(2j+1)}%
\right] \text{.}
\end{equation*}

In our proof of Property 2, we show that the leverage centralities of all
vertices $v_{r}$ are distinct if $m<7$. From a direct calculation on the
formula found in above the leverage centrality satisfies the following
orders. The first three cases were covered at the beginning of Section 5.

\begin{enumerate}
\item If $m=4$, then

$l(v_{t_{m-m}})<$ $l(v_{t_{m-3}+1})<$ $l(v_{t_{m-3}})<l(v_{t_{m-2}+m-2})<$ $%
l(v_{t_{m-2}+1})<$ $l(v_{t_{m-1}+m-1})<$ $l(v_{t_{m}+m})<$ $%
l(v_{t_{m-1}+2})< $ $l(v_{t_{m-2}})<$ $l(v_{t_{m}+3})<$ $l(v_{t_{m-1}+1})<$ $%
l(v_{t_{m}+2})<$ $l(v_{t_{m-1}})<$ $l(v_{t_{m}+1})<$ $l(v_{t_{m}})$.

\item If $m=5$, then

$l(v_{t_{m-m}})<$ $l(v_{t_{m-4}+1})<$ $l(v_{t_{m-4}})<$ $l(v_{t_{m-3}+2})<$ $%
l(v_{t_{m-3}+1})<$ $l(v_{t_{m-2}+3})<$ $l(v_{t_{m-3}})<$ $l(v_{t_{m-2}+2})<$ 
$l(v_{t_{m-1}+m-1})<$ $l(v_{t_{m}+m})<$ $l(v_{t_{m-1}+3})<$ $%
l(v_{t_{m-2}+1})<$ $l(v_{t_{m}+m-1})<$ $l(v_{t_{m-1}+2})<$ $l(v_{t_{m-2}})<$ 
$l(v_{t_{m}+3})<$ $l(v_{t_{m-1}+1})<$ $l(v_{t_{m}+2})<$ $l(v_{t_{m-1}})<$ $%
l(v_{t_{m}+1})<$ $l(v_{t_{m}})$.

\item If $m=6$, then

$l(v_{t_{m-m}})<$ $l(v_{t_{m-5}+1})<$ $l(v_{t_{m-5}})<$ $l(v_{t_{m-4}+2})<$ $%
l(v_{t_{m-4}+1})<$ $l(v_{t_{m-3}+3})<$ $l(v_{t_{m-4}})<$ $l(v_{t_{m-3}+2})<$ 
$l(v_{t_{m-2}+m-2})<$ $l(v_{t_{m-3}+1})<$ $l(v_{t_{m-2}+3})<$ $%
l(v_{t_{m-1}+m-1})<$ $l(v_{t_{m}+m})<$ $l(v_{t_{m-1}+m-2})<$ $%
l(v_{t_{m-2}+2})<$ $l(v_{t_{m-3}})<$ $l(v_{t_{m}+m-1})<$ $l(v_{t_{m-1}+3})<$ 
$l(v_{t_{m-2}+1})<$ $l(v_{t_{m}+m-2})<$ $l(v_{t_{m-1}+2})<$ $l(v_{t_{m}+3})<$
$l(v_{t_{m-2}})<$ $l(v_{t_{m-1}+1})<$ $l(v_{t_{m}+2})<$ $l(v_{t_{m-1}})<$ $%
l(v_{t_{m}+1})<$ $l(v_{t_{m}})$.
\end{enumerate}

This completes the proof.
\end{proof}

We have checked this computationally for all graphs $\underset{m}{{\LARGE %
\times }}P_{n}$ for the first $m\leq 10$ (10 dimensions) and have verified
that there are exactly $\binom{m+2}{2}$ distinct leverage centralities in
each case. We state the general problem for all $m$ as part of Conjecture %
\ref{Conjecture}.

In Theorem \ref{main} we showed that the number of distinct leverage
centralities in $\underset{m}{{\LARGE \times }}P_{n}$ is bounded above by $%
\binom{m+2}{2}$. To show this bound is tight one would need to show that the 
$\binom{m+2}{2}$ leverage centralities are all distinct. By \ref{lemma:
adjacent degrees}, given a vertex $v$ with degree $k$, its neighbors must
have degrees $k-1$, $k$, or $k+1$. Suppose $x$ of $v$'s neighbors have
degree $k-1$ and $y$ of $v$'s neighbors have degree $k+1$. Then the number
of $v$'s neighbors with degree $k$ is $k-x-y$. Hence $l(v)=\frac{1}{k}\left( 
\frac{x}{2k-1}-\frac{y}{2k+1}\right) $. One approach would be to show that $%
\frac{1}{k_{i}}\left( \frac{x_{i}}{2k_{i}-1}-\frac{y_{i}}{2k_{i}+1}\right) $ 
$=\frac{1}{k_{j}}\left( \frac{x_{j}}{2k_{j}-1}-\frac{y_{j}}{2k_{j}+1}\right) 
$ $\Rightarrow k_{i}=k_{j}$, $x_{i}=x_{j}$, and $y_{i}=y_{j}$. However this
appears to be a complex problem.

We have also found that the number of distinct leverage centralities for
graphs of the form $\underset{m}{{\LARGE \times }}P_{n}^{k}$ is linked to
the \textit{polygonal numbers},\textit{\ }which are\textit{\ }numbers that
can be represented by a regular geometrical arrangement of equally spaced
points. For the first few cases, the triangle numbers are given by $P_{2}(m)=%
\binom{m+1}{2}$, the tetrahedral numbers are given by $P_{3}(m)=\binom{m+2}{3%
}$ and the pentalope numbers are given by $P_{4}(m)=\binom{m+3}{4}$. In
general, $P_{k+1}(m)=\binom{m+k}{k+1}$. Since we do not consider a case of a
single vertex, we start all our leverage centrality calculations with the
second polygonal numbers. Hence the general formula translates to $\binom{%
m+k+1}{k+1}$.

Based on our findings for small values of $k$ we pose the following
conjecture.

\begin{conjecture}
\label{Conjecture}Let $n\geq 4k+1$ and $G=$ $\underset{m}{{\LARGE \times }}%
P_{n}^{k}$. Then the number of distinct leverage centralities in $G$ is $%
\binom{m+k+1}{k+1}$.
\end{conjecture}

\bigskip

\textbf{Acknowledgements }The authors are grateful to an anomyous referee
whose comments help improve the quality of this paper. Research was
conducted at the Rochester Institute of Tecnology during the summer of 2015
and was supported by a National Science Foundation Research Experience for
Undergraduates Grant (Award Number: 1358583). Darren Narayan was also
supported by the NSF Grant - STEM\ Real World Applications of Mathematics (
\#1019532).

\end{document}